\documentclass[11pt]{article}

\usepackage{amsfonts, amssymb}
\usepackage{amsmath}
\usepackage{verbatim}
\usepackage{amsthm}
\usepackage{mathrsfs}   
\usepackage{enumerate}

\usepackage{indentfirst}
\usepackage{color}

\newcommand{\p}{\prime}

\newcommand{\n}{\mathbf{n}}
\newcommand{\la}{\lambda}

\begin{document}

\newtheorem{lemma}{Lemma}[section]
\newtheorem{proposition}[lemma]{Proposition}
\newtheorem{theorem}[lemma]{Theorem}
\newtheorem{corollary}[lemma]{Corollary}
\newtheorem{definition}[lemma]{Definition}
\newtheorem{example}[lemma]{Example}
\newtheorem{remark}[lemma]{Remark}
\numberwithin{equation}{section}



\title{Lattice Path Enumeration and Its Applications in Representation Theory}
\date{}
\author {Jianqiang Feng  ~~Wenli Liu  ~~Ximei Bai ~~Zhenheng Li}

\vspace{ -8mm}
\maketitle

\begin{abstract}In this paper, we enumerate lattice paths with certain constraints and apply the corresponding results to develop formulas for calculating the dimensions of submodules of a class of modules for planar upper triangular rook monoids. In particular, we show that the famous Catalan numbers appear as the dimensions of some special modules; we also obtain some combinatorial identities.

\vspace{ 0.2cm} \noindent {\bf Keywords:} Lattice path, module, dimension, Catalan number, rook monoid, order preserving and order decreasing.

\vspace{ 0.2cm} \noindent {\bf 2010 AMS Subject Classification:} 05E10, 05E15, 05A15

\end{abstract}

\baselineskip 14pt
\parskip 1mm
\section{Introduction}

In this paper we find an application of monotonic lattice paths in the representation theory of a planar upper triangular rook monoid $\mathcal{IC}_n$ consisting of order preserving and order decreasing partial maps of $\n=\{1, ..., n\}$. The modules of interest are submodules of $V=\bigoplus_{k=0}^n V_k$, where $V_k$ is a vector space over a field $F$ of characteristic zero generated by a set of elements $v_S$ indexed by the $k$-subsets of $\n$, with the module structure under the action: for $f\in\mathcal{IC}_n$ and $S\subseteq\n$,
\begin{equation*}
f\cdot v_S=
\left\{
  \begin{array}{ll}
    v_{S^\p}, & \hbox{if\; $S\subseteq D(f)$} \\
    0, & \hbox{otherwise,}
  \end{array}
\right.
\end{equation*}
where $D(f)$ is the domain of $f$, $S^\p=f(S)$. We use monotonic lattice path enumeration to describe the dimensions of the submodules of $V$.

Lattice paths under certain constraints have been studied in combinatorics over a long period and many elegant results have emerged (see for example \cite{CCKPRW, Hu, CK}), with applications to problems in probability and statistics: the traditional gambler's ruin, the rank
order statistics for non-parametric testing \cite{KM, Na}, and distributional problems in random walks \cite{Mo}. Recently researchers find applications of lattice paths in commutative algebra; non-intersecting lattice paths are used to describe the Hilbert series of determinantal and Pfaffian rings \cite{HT, Mar}.

The monoids we treat are closely related to the theory of linear algebraic monoids; we are here dealing with a submonoid of the most familiar interesting case of the Renner monoids of reductive monoids   \cite[Scetion 8.5]{R1}. For more information on the Renner monoids, see \cite{P1, R1, R2, So1, LR03, LLC14}.

The organization of the paper is as follows.
In Section 2 we gather some necessary concepts and basic facts about lattice path enumeration and the planar upper triangular rook monoid.

In Section 3 for a given decreasing lattice path from the point $(0, \lambda_1)$ to the point $(k, 0)$ with the height sequence $\lambda_1\ge\cdots\ge\lambda_k$, we develop an iterative formula in Theorem \ref{abg} for enumerating decreasing lattice paths from $(0, \lambda_1)$ to $(k, 0)$ and below the given path. This formula is different from the existing ones for such enumeration, and to our knowledge  it seems new. As consequences, we obtain two combinatorial identities in Corollaries 3.4 and 3.5, of which the latter is related to the Catalan numbers.

In Section 4 we provide three formulas for computing the dimensions of submodules of $V$ via lattice path enumeration.
The first formula stated in Theorem \ref{dim} calculates the dimensions of the submodules $\langle v_S\rangle$ of $V$ generated by a single basis vector $v_S$ where $S$ is a subset of $\n$; this formula is an application of Theorem \ref{abg}.
As a result of the application of \cite[Theorem 10.7.1]{CK}, the second formula for computing the dimension of the submodule $\langle v_S\rangle$ is given in Theorem \ref{dim2}. It turns out that the famous Catalan numbers appear as the dimensions of some submodules $\langle v_S\rangle$.
The third formula described in Theorem \ref{dimEvery} calculates the dimension of every submodule of $V$.
For developing these formulas we have investigated the properties of submodules of $V$, and in this paper we only show the properties needed for establishing our formulas. We find that the submodule $\langle v_S\rangle$ is a direct sum of the vector spaces $Fv_T$ where $T$ ranges over all subsets less that or equal to $S$. We then show that every submodule of $V$ is cyclic and contains a unique reduced generator. Furthermore we conclude that two submodules equal if and only if they have the same reduced generator.

{\bf Acknowledgement} {We would like to thank Dr. M. Can for useful email communications and Dr. R. Koo for valuable comments.}

\section{Preliminaries}\label{Pre}

We gather necessary definitions and basic facts about lattice paths and planar upper triangular rook monoids.

\subsection{Lattice Paths}

A lattice path $\boldsymbol v$ is a sequence of finite lattice points $v_1, \ldots, v_k$ of $\mathbb Z^2$, with $v_1$ the starting point and $v_k$ the ending point of the path. The vectors $\overrightarrow{v_1 v_2}, \overrightarrow{v_2 v_3},  \ldots, \overrightarrow{v_{l-1} v_k}$ are referred to as the steps of the path.
A lattice path is decreasing if each step is either $(1, 0)$ or $(0, -1)$, while a lattice path is increasing if each step is either $(1, 0)$ or $(0, 1)$. A lattice path is \emph{monotonic} if it is decreasing or increasing. The lattice path from $(0, 4)$ to $(5, 0)$ in the left figure below is decreasing, and the one from $(0, 0)$ to $(5, 4)$ in the right is increasing.

\newpage
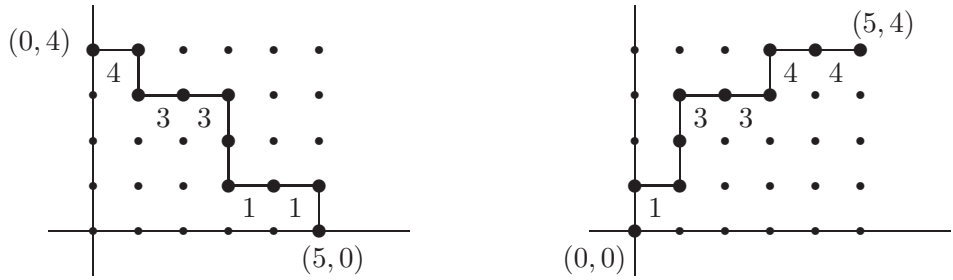
\begin{center}
    \begin{picture}(90, 30)(110, 0)
    \setlength{\unitlength}{1.2mm}

    \put(0, 5){\line(1, 0){40}}
    \put(5, 0){\line(0, 1){30}}

    \put(5, 5){\circle*{1}}
    \put(5, 10){\circle*{1}}
    \put(5, 15){\circle*{1}}
    \put(5, 20){\circle*{1}}
    \put(5, 25){\circle*{1.5}}

    \put(10, 5){\circle*{1}}
    \put(10, 10){\circle*{1}}
    \put(10, 15){\circle*{1}}
    \put(10, 20){\circle*{1.5}}
    \put(10, 25){\circle*{1.5}}

    \put(15, 5){\circle*{1}}
    \put(15, 10){\circle*{1}}
    \put(15, 15){\circle*{1}}
    \put(15, 20){\circle*{1.5}}
    \put(15, 25){\circle*{1}}

    \put(20, 5){\circle*{1}}
    \put(20, 10){\circle*{1.5}}
    \put(20, 15){\circle*{1.5}}
    \put(20, 20){\circle*{1.5}}
    \put(20, 25){\circle*{1}}

    \put(25, 5){\circle*{1}}
    \put(25, 10){\circle*{1.5}}
    \put(25, 15){\circle*{1}}
    \put(25, 20){\circle*{1}}
    \put(25, 25){\circle*{1}}

    \put(30, 5){\circle*{1.5}}
    \put(30, 10){\circle*{1.5}}
    \put(30, 15){\circle*{1}}
    \put(30, 20){\circle*{1}}
    \put(30, 25){\circle*{1}}

    \put(5, 25){\line(1, 0){5}}
    \put(10,25){\line(0, -1){5}}
    \put(10,20){\line(1, 0){10}}
    \put(20,20){\line(0, -1){10}}
    \put(20,10){\line(1, 0){10}}
    \put(30,10){\line(0, -1){5}}

    \put(-4.5, 25){$(0, 4)$}

    \put(6.5, 21.5){$4$}
    \put(12, 16.5){$3$}
    \put(16.5, 16.5){$3$}
    \put(26.6, 6.5){$1$}
    \put(21.5, 6.5){$1$}

    \put(28, 1){$(5, 0)$}


    \put(60, 5){\line(1, 0){40}}
    \put(65, 0){\line(0, 1){30}}

    \put(65, 5){\circle*{1.5}}
    \put(65, 10){\circle*{1.5}}
    \put(65, 15){\circle*{1}}
    \put(65, 20){\circle*{1}}
    \put(65, 25){\circle*{1}}

    \put(70, 5){\circle*{1}}
    \put(70, 10){\circle*{1.5}}
    \put(70, 15){\circle*{1.5}}
    \put(70, 20){\circle*{1.5}}
    \put(70, 25){\circle*{1}}

    \put(75, 5){\circle*{1}}
    \put(75, 10){\circle*{1}}
    \put(75, 15){\circle*{1}}
    \put(75, 20){\circle*{1.5}}
    \put(75, 25){\circle*{1}}

    \put(80, 5){\circle*{1}}
    \put(80, 10){\circle*{1}}
    \put(80, 15){\circle*{1}}
    \put(80, 20){\circle*{1.5}}
    \put(80, 25){\circle*{1.5}}

    \put(85, 5){\circle*{1}}
    \put(85, 10){\circle*{1}}
    \put(85, 15){\circle*{1}}
    \put(85, 20){\circle*{1}}
    \put(85, 25){\circle*{1.5}}

    \put(90, 5){\circle*{1}}
    \put(90, 10){\circle*{1}}
    \put(90, 15){\circle*{1}}
    \put(90, 20){\circle*{1}}
    \put(90, 25){\circle*{1.5}}

    \put(65, 10){\line(1, 0){5}}
    \put(70,20){\line(0, -1){10}}
    \put(70,20){\line(1, 0){10}}
    \put(80,20){\line(0, 1){5}}
    \put(80,25){\line(1, 0){10}}
    \put(65,10){\line(0, -1){5}}

    \put(57, 1){$(0, 0)$}

    \put(66.5, 6.5){$1$}
    \put(71.5, 16.5){$3$}
    \put(76.5, 16.5){$3$}
    \put(81.5, 21.5){$4$}
    \put(86.5, 21.5){$4$}

    \put(89, 27){$(5, 4)$}
    \end{picture}
\end{center}

From now on we assume that all the lattice paths under consideration are monotonic and lie in the closed first quadrant, the union of the first quadrant, the nonnegative $x$-axis, and the nonnegative $y$-axis. We further assume that the starting points of the paths are all on the nonnegative $y$-axis.

The heights $\lambda_i$ of all the horizontal steps of a monotonic lattice path form a finite sequence $(\lambda_1, \ldots,  \lambda_k)$ of nonnegative integers, which is uniquely determined by the lattice path, and is called the \emph{height sequence} of the path. If the path is decreasing (resp. increasing)  its height sequence is denoted by $\lambda_1 \ge \ldots \ge  \lambda_k$ (resp.   $\lambda_1 \le \ldots \le  \lambda_k$).
The height sequence of the path in the left figure above is $4\ge 3\ge 3\ge 1\ge 1$, and that of the path in the right is $1\le 3\le 3\le 4\le 4$.

%
%
Note that two different lattice paths may have the same height sequence. But for a given finite decreasing sequence $\lambda_1 \ge \ldots \ge  \lambda_k$ of nonnegative integers, there exists a unique decreasing lattice path in the closed first quadrant from the point $(0, \lambda_1)$ to the point $(k, 0)$, whose height sequence is the given sequence.
Similarly, for a given finite increasing sequence $a_1 \le \ldots \le  a_k$ of nonnegative integers there exists a unique increasing lattice path in the closed first quadrant from the point $(0, 0)$ to the point $(k, a_k)$, whose height sequence is the given sequence.

It is convenient to identify a decreasing lattice path from the point $(0, \lambda_1)$ to the point $(k, 0)$ with its height sequence $\lambda_1 \ge \ldots \ge  \lambda_k$, and identify an increasing lattice path from the point $(0, 0)$ to the point $(k, a_k)$ with its height sequence $a_1 \le \ldots \le a_k$, and we often do so without mentioning it further.
We refer the reader to \cite{Hu, CK} and the references cited there for a comprehensive survey of lattice path enumeration.

\subsection{Planar Upper Triangular Rook Monoids}

An {\em injective partial map} $f$ of $\mathbf{n}$ is a one-to-one map of a subset $D(f)$ of $\mathbf{n}$ onto a subset $R(f)$ of $\mathbf{n}$ where $D(f)$ is the domain of $f$ and $R(f)$ is the range of $f$. We agree that there is a map with empty domain and range and call it 0 map.

We can represent an injective partial map by an $n\times n$ matrix, where the entry in the
$i$th row and the $j$th column is 1 if the map takes $j$ to $i$, and is 0 otherwise, such a matrix is named a rook matrix, a matrix with at most one 1 in each row and each column.
For example, the map $\sigma$ given below is an injective partial map of $\n=\{1, 2, 3, 4\}$,

{
\small
\begin{eqnarray*}
\sigma&=& \left(
            \begin{array}{cccc}
              1 & 2 & 3 & 4 \\
              1 & \times & 2 & 3 \\
            \end{array}
          \right)
 \\
   &=& \left(
         \begin{array}{cccc}
           1 & 0 & 0 & 0 \\
           0 & 0 & 1 & 0 \\
           0 & 0 & 0 & 1 \\
           0 & 0 & 0 & 0
         \end{array}
       \right)~.
\end{eqnarray*}
}

The {\em rook monoid} $R_n$ is the monoid of
injective partial maps from $\mathbf{n}$ to $\mathbf{n}$, whose operation is the composition of partial maps and the identity element is the identity map of $\n$. Since elements of $R_n$ are not necessarily invertible, $R_n$ is not a group. The map with empty domain and empty range behaves as a zero element. Identifying an injective partial map with its associated rook matrix, $R_n$ can be regarded as the monoid consisting of all the rook matrices of size $n$. The structures and representations of the rook monoid are intensively studied \cite{CP, GM3, Ho, M2, So2}; the generating functions of $R_n$ and their connections to Laguerre polynomials are found in \cite{BRR}.

We can write an injective partial map $f$ of $\mathbf{n}$ in
2-line notation by writing the numbers $s_1,\dots, s_k$ in the
top line if $D(f)=\{s_1,\dots, s_k\}$, and then below each
number we write its image.

An injective partial map from $\mathbf{n}$ to $\mathbf{n}$ is {\em order preserving} if whenever $a<b$ in the domain of the map, then $f(a)<f(b)$. It is easily seen that an injective partial map $f$ is order preserving if and only if the graph obtained from the 2-line notation of $f$ by joining all defined $f(a)$ in the range of the map to $a$ is a planar graph, which justifies the name in the following definition.
\begin{definition}
    The {\em planar rook monoid} is the monoid of {\em order preserving} injective partial maps from $\mathbf{n}$ to $\mathbf{n}$.
\end{definition}
Obviously, a planar rook monoid is a submonoid of $R_n$. The structure and representation of the planar rook monoid is studied in \cite{FHH}.

An injective partial map is called {\em order decreasing} if for all $a$ in the domain of the map, we have $f(a)\leq a$. Clearly, an injective partial map is order decreasing if and only if its matrix form is an upper triangular rook matrix, which motivates the name in the following definition.

\begin{definition}
The {\em planar upper triangular  rook monoid}, denoted by $\mathcal{IC}_n$, is the monoid of order preserving, order decreasing injective partial maps from $\mathbf{n}$ to $\mathbf{n}$.
\end{definition}
The notation $\mathcal{IC}_n$ for the planar upper triangular rook monoid is standard in semigroup theory, see for example \cite[Chapter 14]{GM3}.

\section{Lattice path enumeration}\label{LPE}
Let $\boldsymbol v$ be a decreasing lattice path in the closed first quadrant with the starting point $(0, \lambda_1)$, the ending point $(k, 0)$, and the height sequence $\lambda_1\ge\ldots\ge \lambda_k$. We say that a decreasing lattice path $\boldsymbol u$ with the height sequence $\mu_1\ge\ldots\ge\mu_l$ is \emph{below} $\boldsymbol v$ if it is from $(0, \lambda_1)$ to $(k, 0)$ (hence $l=k$) and  $0\le\mu_i\le\lambda_i$ for $i=1,\ldots,k.$ In particular, if $\boldsymbol u$ is below $\boldsymbol v$, then they share the same starting point and the same ending point.
The concept `below' for increasing lattice paths is defined similarly.

The purpose of this section is to calculate the number $d_k$ of all the decreasing lattice paths below $\boldsymbol v$.
The next lemma is immediate.
\begin{lemma}\label{dkseq}
    The number $d_k$ is equal to the number of decreasing sequences $\mu_1\ge\cdots\ge\mu_k$ of integers such that $0\le\mu_i\le\lambda_i, \,i=1, \ldots, k.$
\end{lemma}

We give a formula for calculating $d_k$ in Theorem \ref{abg}.
An example is useful to illustrate the idea of its proof. Let $k=2$ and $\lambda_1\ge\lambda_2$ be $4\ge 2$. We first fix $\mu_2 = 2~ (= \lambda_2)$; there are 3 sequences $\mu_1\ge\mu_2$ with $\mu_1\le\lambda_1: 4\ge 2;~ 3\ge 2; ~2\ge 2$. We then fix $\mu_2 = 1$; there are 4 such sequences: $4\ge 1; ~3\ge 1;  ~2\ge 1; ~1\ge 1$. We now fix $\mu_2 = 0$; there are 5 such sequences $: 4\ge 0; ~3\ge 0; ~2\ge 0; ~1\ge 0; ~0\ge 0$. Lemma \ref{dkseq} indicates $d_2=3+4+5 =12.$

From now on, we agree that if $a>b$, the empty sum $\sum_{i=a}^b \square_i = 0$.

\begin{theorem}\label{abg}
Let $\boldsymbol v$ be a given decreasing lattice path in the closed first quadrant with the starting point $(0, \lambda_1)$, the ending point $(k, 0)$, and the height sequence $\lambda_1 \ge \ldots \ge \lambda_k$ of nonnegative integers. Then the number $d_k$ of decreasing lattice paths below $\boldsymbol v$ is given by $d_1= \lambda_1+1$, and for $k\ge 2$ by

{\small
\[
d_k = \sum^{k-1}_{i=1}\bigg[{\la_i+k-i+1 \choose k+1-i} -{\la_i-\la_k+k-i \choose k+1-i}\bigg]\gamma_i -\sum^{k-2}_{i=1}(\la_k+1) {\la_i-\la_{k-1}+k-i-1 \choose k-i}\gamma_i~,
\]
}
where
\begin{eqnarray}\label{gamma}
\gamma_1 = 1 \quad\text{and}\quad\gamma_j =-\sum^{j-2}_{i=1}
          {\la_i-\la_{j-1}+j-i-1 \choose j-i}\gamma_i~\quad\text{for } j\ge 2.
\end{eqnarray}
\end{theorem}
\begin{proof}
To find $d_k$, by Lemma \ref{dkseq} it suffices to compute the number of the sequences $\mu_1\geq \dots\geq \mu_k$ of nonnegative integers with $\mu_i\leq \lambda_i$ for $i=1,\dots,k.$
If $k=1$, clearly $d_1 = \lambda_1 + 1$.

If $k\ge 2$, let $2\le j\le k$. For each fixed nonnegative integer $\mu\le\la_j$,
denote by $\alpha_j(\mu)$ the number of sequences of nonnegative integers
\begin{equation}\label{aSj}
  \mu_1\geq \dots\geq \mu_{j-1}\ge\mu\quad\text{with}\quad \mu_i\leq \lambda_i\quad\text{for}\quad i=1,\dots,j-1~.
\end{equation}
We calculate $\alpha_j(\mu)$ iteratively on $j$,
and the required number $d_k = \sum_{\mu=0}^{\lambda_k}\alpha_k(\mu)$.

Let $\xi_j=\la_j-\mu$. Then $0\leq \xi_j\leq\la_j$. Our aim now is to prove
\begin{equation}\label{alphaj}
  \alpha_j(\mu) = \beta_j + \gamma_j~,
\end{equation}
where
$\beta_j=\sum^{j-1}_{i=1}{\la_i-\la_j+\xi_j+j-i \choose j-i}\gamma_i$ and $\gamma_j = -\sum^{j-2}_{i=1}{\la_i-\la_{j-1}+j-i-1 \choose j-i}\gamma_i$ with $\gamma_1=1$.
Notice that $\alpha_j(\mu)$ is a sum of two numbers $\beta_j$ and $\gamma_j$, of which $\gamma_j$ depends on $\la_1,\dots,\la_{j-1}$, whereas $\beta_j$ depends on $\la_1,\dots,\la_j$ and $\xi_j$.

We use induction on $j$ to prove (\ref{alphaj}) for $2 \le j\le k$. If $j=2$, for each fixed nonnegative integer $\mu\le\lambda_2$
we have $\xi_2=\la_2-\mu$ and $0\leq \xi_2\leq\la_2$. Let $\xi_1=\la_1-\mu_1$. To ensure that (\ref{aSj}) holds for this case, namely $\mu_1\geq\mu$ and $\mu_1\leq\lambda_1$, we must have $0\leq \xi_1\leq\la_1-\la_2+\xi_2$, and conversely. So
\begin{equation*}
    \alpha_2(\mu)= \la_1-\la_2+\xi_2+1 = \beta_2+\gamma_2
\end{equation*}
where $\beta_2=\la_1-\la_2+\xi_2+1$ and $\gamma_2=0$, and this is (\ref{alphaj}) for $j=2$.

Suppose (\ref{alphaj}) holds for $j=l$ with $2\leq l\leq k-1$, that is, for each fixed nonnegative integer $\mu\le\lambda_l$
we have $\xi_l=\la_l-\mu$ with $0\leq \xi_l\leq\la_l$, and the number of sequences $\mu_1\geq\dots\geq\mu_{l-1}\geq\mu$ with $\mu_i\leq\lambda_i$ for $i=1,\ldots,l-1$ is
\begin{equation}\label{hypothesis}
  \alpha_l(\mu) = \beta_l + \gamma_l~,
\end{equation}
where
$\beta_l=\sum^{l-1}_{i=1}{\la_i-\la_l+\xi_l+l-i \choose l-i}\gamma_i$ and $\gamma_l = -\sum^{l-2}_{i=1}{\la_i-\la_{l-1}+l-i-1 \choose l-i}\gamma_i$~.

We now prove (\ref{alphaj}) for $j=l+1$. For a fixed nonnegative integer $\nu\le\lambda_{l+1}$ we have $\xi_{l+1} = \la_{l+1}-\nu$ with $0\leq \xi_{l+1}\leq\la_{l+1}$. Let $\mu=\la_l-\xi_l$. To ensure that the condition (\ref{aSj})
\[
     \mu_1\geq\dots\geq\mu_{l-1}\geq\mu\ge\nu\quad\text{ with}\quad\mu_i\le\lambda_i,\, i=1,\ldots,{l-1}\text{ and } \mu\le\lambda_l
\]
holds here, we must have $0\leq \xi_l\leq\rho_l$ where $\rho_l=\la_l-\la_{l+1}+\xi_{l+1}$, and conversely. Adding all $\alpha_l(\mu)$ up for $\nu\le\mu\le\lambda_l$ and using the induction hypothesis (\ref{hypothesis}), we obtain
\begin{eqnarray}
\nonumber\alpha_{l+1}(\nu)&=&\sum_{\mu = \nu}^{\lambda_{l}} \alpha_l(\mu)  = \sum_{\mu = \nu}^{\lambda_{l}} (\beta_l + \gamma_l)\\
          &=&\sum^{\rho_l}_{\xi_l=0}\sum^{l-1}_{i=1}{\la_i-\la_l+\xi_l+l-i \choose l-i}\gamma_i+\sum^{\rho_l}_{\xi_l=0}\gamma_l \label{second}\\
\nonumber &=&\sum^{l-1}_{i=1}\left\{{\la_i-\la_{l+1}+\xi_{l+1}+(l+1)-i \choose l+1-i}\gamma_i
    -{\la_i-\la_l+l-i \choose l+1-i}\gamma_i\right\}\\
          &&\qquad\qquad\qquad\qquad\qquad+{\la_l-\la_{l+1}+\xi_{l+1}+1\choose 1}\gamma_l \label{alp1}\\
\nonumber &=&\sum^l_{i=1}{\la_i-\la_{l+1}+\xi_{l+1}+(l+1)-i \choose l+1-i}\gamma_i-\sum^{l-1}_{i=1}{\la_i-\la_l+l-i \choose l+1-i}\gamma_i\\
\nonumber &=&\beta_{l+1} + \gamma_{l+1},
\end{eqnarray}
where
\begin{eqnarray*}\label{}
\beta_{l+1}&=&\sum^l_{i=1}
    {\la_i-\la_{l+1}+\xi_{l+1}+(l+1)-i \choose l+1-i}\gamma_i ~,\\
\gamma_{l+1}&=&-\sum^{l-1}_{i=1}{\la_i-\la_l+l-i \choose
l+1-i}\gamma_i ~.
\end{eqnarray*}
Here we have made use of the identity $\sum_{z=a}^{a+b-1}\binom{z}{p} = \binom{a+b}{p+1} - \binom{a}{p+1}$ in which $a, b, p$ are natural numbers to obtain (\ref{alp1}) from (\ref{second}) by assigning $a=\lambda_i - \lambda_l + l - i,\, b=\lambda_l - \lambda_{l+1} + \xi_{l+1} + 1$ and $p=l-i\ge 1$.
Therefore, (\ref{alphaj}) is valid for $j=l+1$, and we complete the proof of (\ref{alphaj}) by induction.

We are now able to calculate the number $d_k$ for $k\ge 2$ by summing all $\alpha_k(\mu)$ in (\ref{alphaj}) up where $\mu$ runs from $0$ to $\la_k$, yielding
\begin{eqnarray*}\label{dimlambda3}
\nonumber d_k &=& \sum^{\la_k}_{\mu=0}\alpha_k(\mu) \\
\nonumber   &=&\sum^{\la_k}_{\xi_k=0}\sum^{k-1}_{i=1}
                {\la_i-\la_k+\xi_k+k-i \choose k-i}\gamma_i-\sum^{\la_k}_{\xi_k=0}\sum^{k-2}_{i=1}
                {\la_i-\la_{k-1}+k-i-1 \choose k-i}\gamma_i\\
\nonumber   &=&\sum^{k-1}_{i=1}{\la_i+k-i+1 \choose k+1-i}\gamma_i -\sum^{k-1}_{i=1}{\la_i-\la_k+k-i \choose k+1-i}\gamma_i \\
            &&\qquad\qquad -\sum^{k-2}_{i=1}(\la_k+1) {\la_i-\la_{k-1}+k-i-1 \choose k-i}\gamma_i~,
\end{eqnarray*}
which is the desired result.
\end{proof}

The next result is well-known, and is a special case of \cite[Theorem 10.7.1]{CK} enumerating the number of increasing lattice paths below a given increasing lattice path. Note that \cite[Theorem 10.7.1]{CK} was initially obtained in \cite{Kr} using recurrence relations. Recall that for a given increasing lattice path $\boldsymbol v$ in the closed first quadrant from $(0, 0)$ to $(k, a_k)$ with the height sequence $a_1 \le \ldots \le a_k$, we say that an increasing lattice path $\boldsymbol u$ with the height sequence $b_1\le\ldots\le b_l$ is \emph{below} $\boldsymbol v$ if the two paths share the same starting point $(0, 0)$ and the same ending point $(k, a_k)$ (hence $l=k$) and $0\le b_i\le a_i$ for $i=1,\ldots,k.$

\begin{proposition}\label{ck}
Let $\boldsymbol v$ be a given increasing lattice path in the closed first quadrant from $(0, 0)$ to $(k, a_k)$ with the height sequence $a_1 \le \ldots \le a_k$. Then the number $l_k$ of increasing lattice paths below $\boldsymbol v$ is
   \[
        l_k = \det_{1\le i, j \le k} \bigg(\binom{a_{i} + 1}{j-i+1}\bigg)~.
   \]
\end{proposition}

Connecting Theorem \ref{abg} to Proposition \ref{ck}, we obtain a combinatorial identity.
\begin{corollary}\label{newId}
    For a sequence $\lambda_1\ge\cdots\ge\lambda_k$ of nonnegative integers, we have
    \begin{eqnarray*}
\nonumber \det_{1\le i, j \le k} \binom{\lambda_i + 1}{i-j+1}
            &=&\sum^{k-1}_{i=1}{\la_i+k-i+1 \choose k+1-i}\gamma_i -\sum^{k-1}_{i=1}{\la_i-\la_k+k-i \choose k+1-i}\gamma_i \\
            &&\qquad\qquad -\sum^{k-2}_{i=1}(\la_k+1) {\la_i-\la_{k-1}+k-i-1 \choose k-i}\gamma_i~,
\end{eqnarray*}
where $\gamma_i$ is given in (\ref{gamma}).
\end{corollary}
\begin{proof}
    Let $a_i = \lambda_{k-i+1}$. Then $a_1 \le \cdots \le a_k$. It follows from Proposition \ref{ck} that the number $l_k$ of increasing lattice paths below $a_1 \le \cdots \le a_k$ is
   \[
        l_k = \det_{1\le i, j \le k} \bigg(\binom{\lambda_{k-i+1} + 1}{j-i+1}\bigg)~.
   \]

Let $D$ be the anti-diagonal matrix of size $k$ with $1$ for each entry on the anti-diagonal and all other entries 0. If $A=(a_{ij})$ is any matrix of size $k$, then $DAD = (a_{_{k-i+1,\, k-j+1}})$ and $\det A = \det DAD.$ Taking $A = \big(\binom{\lambda_{k-i+1} + 1}{j-i+1}\big)$, we find that
$
    D A D = \big(\binom{\lambda_i + 1}{i-j+1}\big).
$
It follows that
\[
    l_k = \det_{1\le i, j \le k} \bigg(\binom{\lambda_i + 1}{i-j+1}\bigg)~.
\]

By symmetry, this number equals the number $d_k$ of decreasing lattice paths below the lattice path from $(0, \lambda_1)$ to $(k, 0)$ with the height sequence $\lambda_1\ge\cdots\ge\lambda_k$.
Using Theorem \ref{abg}, we obtain the required result.
\end{proof}

We have the combinatorial identity below for the Catalan number 
$c_n=\frac{1}{n+1}\binom{2n}{n}$.
To our knowledge, the identity is new.
\begin{corollary}\label{comId}
If $k\geq 2$, then
\begin{equation*}
c_{k+1}=\sum^{k-1}_{i=1}{2(k-i+1) \choose k+1-i}\gamma_i
  -\sum^{k-1}_{i=1}{2(k-i) \choose k+1-i}\gamma_i-\sum^{k-2}_{i=1}2
          {2(k-i-1) \choose k-i}\gamma_i~,
\end{equation*}
where $\gamma_1=1$ and for $2\leq i\leq k$,
\begin{equation*}
\gamma_i =-\sum^{i-2}_{j=1}{2(i-j-1) \choose i-j}\gamma_j~.
\end{equation*}
\end{corollary}
\begin{proof}
Let $(\la_1\geq\la_2\geq\cdots\geq\la_k)=(k \geq(k-1)\geq\cdots\geq 1)$. We find that the number of sequences $\mu_1\ge\mu_2\ge\ldots\ge\mu_k$ of nonnegative integers such that $\mu_i\leq\la_i$ for $1\leq i\leq k$ is the Catalan number $c_{k+1}$.
Simplifying the formula for $d_k$ in Theorem \ref{abg}, we complete the proof.
\end{proof}

\begin{corollary}
Let $(\la_1\geq\dots\geq\la_k\geq 0)$ be a given partition
of some nonnegative integer. Then the number of distinct Young diagrams, with each row having equal or fewer boxes than the row above, obtained from the Young diagram of $\la$ by removing zero or more boxes from the rows is $d_k$.
\end{corollary}
\begin{proof}
  It is easily seen that the number of the desired distinct Young diagrams obtained from the Young diagram of $\la$ by removing zero or more boxes from the rows is equal to the number of sequences $\mu_1\geq\dots\geq\mu_k\geq 0$ such that $\mu_i\leq \la_i$ for all $i=1,2,\dots,k$.  The result follows from Lemma \ref{dkseq} and Theorem \ref{abg}.
\end{proof}

\section{Dimensions of modules over $\mathcal{IC}_n$}\label{dimension}

Our aim of this section is to apply Theorem \ref{abg} for enumerating lattice paths to calculate the dimensions of submodules for $\mathcal{IC}_n$. To this end we need some preparations to describe precisely the structure of the modules involved; our results go a little deeper and wider than just for calculating the dimensions.

\subsection{Properties of modules over $\mathcal{IC}_n$}

A vector space $V$ over a field $F$ of characteristic $0$ is called an $\mathcal{IC}_n$-module if $\mathcal{IC}_n$ acts on $V$ satisfying, for all $f, f_1, f_2\in \mathcal{IC}_n$, $u, v\in V$, and $\lambda\in F$,
\begin{eqnarray*}
  f\cdot (u+v) &= f\cdot u + f\cdot v, \quad\quad\quad f_1\cdot (f_2\cdot u) &= (f_1f_2)\cdot u, \\
  f\cdot (\lambda u)  &= \lambda (f\cdot u),~\quad\quad\quad\quad\quad\quad\quad 1\cdot u &= u.
\end{eqnarray*}

From now on, $V$ denotes a vector space with a basis
$
   \mathcal{B} = \{v_S\mid S\subseteq\n\}
$
indexed by all the subsets of $\n$. Then $V=\bigoplus_{S\subseteq\n} Fv_S$ as subspaces is an $\mathcal{IC}_n$-module with respect to the following action: for $f\in \mathcal{IC}_n$ and $S\subseteq\n$,
\begin{equation*}
f\cdot v_S=
\left\{
  \begin{array}{ll}
    v_{S^\p}, & \hbox{if\; $S\subseteq D(f)$} \\
    0, & \hbox{otherwise,}
  \end{array}
\right.
\end{equation*}
where $S^\p=\{f(s_1),\dots,f(s_k)\}$ if $S=\{s_1,\dots,s_k\}$. For $0\le k \le n$, let
$$
    V_k=\mathrm{span}\{v_S\in\mathcal{B}\mid k=|S|\}.
$$
Then $V=\bigoplus^n_{k=0}V_k$ is a direct sum of $\mathcal{IC}_n$-submodules.

Every module under consideration is an $\mathcal{IC}_n$-module over $F$, unless otherwise stated.
To describe the $\mathcal{IC}_n$-module structure of $V_k$ and $V$, we
define a partial order on the power set of $\n$. For any $k$-subsets
$S=\{s_1<\dots<s_k\}$ and $T=\{t_1<\dots<t_k\}$ of $\n$, define
\begin{equation*}
    T\leq S \quad \Leftrightarrow \quad t_i\leq s_i \quad\text{for all}\quad i\in \mathbf{k}~,
\end{equation*}
and a $k$-subset is not comparable to any $l$-subset if $k\ne l$.

For $v\in V$ we use $\langle v \rangle$ to denote the cyclic submodule of $V$ generated by $v$. If $S$ is a $k$-subset of $\n$, then $\langle v_S \rangle$ is a submodule of $V_k$. Indeed, for any $f\in \mathcal{IC}_n$ if $S\subseteq D(f)$ then $f(S)$ is a $k$-subset, so $f\cdot v_S = v_{f(S)}\in V_k$; if $S$ is not a subset of $D(f)$ then $f\cdot v_S = 0\in V_k$.  Some further properties of the module $\langle v_S \rangle$ are described in the next result.
\begin{lemma}\label{mod1} Let $S,T$ be $k$-subsets of $\n$.

{\rm(1)} $\langle v_S \rangle = \bigoplus_{S'\subseteq\n,\, S'\leq S}Fv_{S'}$ as vector spaces. In particular, $V_k=\langle v_{\{n-k+1,\, \ldots,\, n\}} \rangle$.

{\rm(2)} $\langle v_T \rangle\subseteq \langle v_S \rangle$ if and only if $T\leq S$.

{\rm(3)} $\langle v_S \rangle\cap \langle v_T \rangle = \langle v_{S\wedge T} \rangle$, where $S\wedge T$ is the greatest lower bound of $S$ and $T$.
\end{lemma}
\begin{proof}

To prove (1) notice that two subsets $S'\leq S$ if and only if $S=D(f)$ and $S'=R(f)$ for a unique $f\in \mathcal{IC}_n$. Let $S'\le S$. Then $v_{S'} = f\cdot v_S\in \langle v_S \rangle$. Hence $\bigoplus_{S'\subseteq\n,\, S'\leq S}Fv_{S'}$ is included in $\langle v_S \rangle$. Conversely, let $x=g\cdot v_S \ne 0$ for some $g\in \mathcal{IC}_n$. We have $S\subseteq D(g)$, $g(S)\le S$, and hence $x=v_{g(S)}\in \bigoplus_{S'\subseteq\n,\, S'\leq S}Fv_{S'}$. The second part of (i) is now clear.

The proof of (2) follows from (1) since $\{T'\mid T'\subseteq\n,\, T'\leq T\}\subseteq\{S'\mid S'\subseteq\n,\, S'\leq S\}$ if and only if $T\le S$.

To prove (3) let $g\cdot v_S = h\cdot v_T\ne 0$ for some $g, h\in \mathcal{IC}_n$. Then $g(S)=h(T)$. Suppose
\[
    S = \{s_1< \ldots < s_k\}\quad\text{and}\quad T = \{t_1< \ldots < t_k\}~.
\]
Then
$
    S \wedge T =\{\min(s_1, t_1),\, \ldots,\, \min(s_k, t_k)\},
$
and $g(s_i)=h(t_i)$. We define $f\in \mathcal{IC}_n$ with $D(f)=S \wedge T $ and $R(f)=g(S)$ by
$
    f(\min(s_i,\, t_i)) =  g(s_i),
$
where $1\le i\le k$. Then $g\cdot v_S = f\cdot v_{S\wedge T} \in \langle v_{S\wedge T} \rangle$, and hence $\langle v_S \rangle\cap \langle v_T \rangle \subseteq \langle v_{S\wedge T} \rangle$.
Conversely, for any given $0\ne f\cdot v_{S\wedge T} \in \langle v_{S\wedge T} \rangle$ define $g(s_i)=h(t_i)= f(\min(s_i,\, t_i))$ for $1\le i\le k$. Then $f\cdot v_{S\wedge T} = g\cdot v_S=h\cdot v_T\in \langle v_S \rangle\cap \langle v_T \rangle$. The proof of (3) is complete.
\end{proof}

    Let $v=\sum_{S\subseteq \n}\lambda_Sv_S, \lambda_S\in F$ be a vector of $V$. The {\em support} of $v$ is defined to be
    \[
        {\rm supp}(v) = \{S\subseteq\n\mid \lambda_S \ne 0\}~.
    \]
%
\begin{definition}\label{redGen}
    A vector of the form $w=\sum_{S\in {\rm supp}(w)}v_S\in V$ is called a {\em reduced generator} of a submodule $W$ of $V$ if $W=\langle w \rangle$ and $W$ cannot be generated by any other vector whose support contains fewer elements than {\rm supp(}$w${\rm )}. We agree that $0$ is the reduced generator of the zero submodule.
\end{definition}
The next proposition gives some properties of submodules of $V$.
\begin{proposition}\label{cyclic} Let $v=\sum_{S\in\,{\rm supp(}v{\rm )}}\lambda_Sv_S\in V$.

{\rm (1)} If $S$ is in {\rm supp}$(v)$, then $v_S\in \langle v \rangle$~.

{\rm (2)} $\langle v \rangle = \bigoplus_{T\in \mathcal{P}(v)} Fv_T$ as subspaces, where $\mathcal{P}(v) = \bigcup_{S\in {\rm supp}(v)}\{T\subseteq\n\mid T\le S\}$.

{\rm (3)} Every submodule of $V$ is cyclic and contains a unique reduced generator.
\end{proposition}
\begin{proof} To prove (1) let $\min \big\{\,|S| \,\big|\, S \in {\rm supp}(v)\big\}=r$. Then there exists an $r$-subset $T=\{t_1<\cdots < t_r\}\subseteq\n$ such that $T\in {\rm supp}(v)$; if $r=0$, then $T=\emptyset$. Let $f\in \mathcal{IC}_n$ such that $D(f)=R(f)=T$. By the choice of $r$, for every $S\in{\rm supp}(v)$ with $S\neq T$, there is at least one $s\in S$ such that $s\notin T$, so $f\cdot v_S=0$. Hence
$$
    f\cdot v=f\cdot \sum_{S\in\,{\rm supp(}v{\rm )}}\lambda_Sv_S=\sum_{S\in\,{\rm supp(}v{\rm )}}\lambda_S(f\cdot v_S)=\lambda_Tv_T~.
$$
Thus $v_T\in \langle v \rangle$ since $\lambda_T\neq 0$. It is easily seen that
$$
    \sum_{S\in\,{\rm supp(}v{\rm )}\atop |S|>r}\lambda_Sv_S=v-\sum_{S\in\,{\rm supp(}v{\rm )}\atop |S|=r}\lambda_Sv_S\in \langle v \rangle~.
$$
Applying the above procedure to
$
    \sum_{S\in\,{\rm supp(}v{\rm )},\,|S|>r}\lambda_Sv_S
$
and iteratively using this procedure, if needed, we get $v_S\subseteq \langle v \rangle$ for all $S\in \text{supp}(v)$. The proof of (1) is complete.

From (1) and Lemma \ref{mod1} (1), we have
\begin{eqnarray*}
   \langle v \rangle&=& \sum_{S \in\text{\rm supp}(v)}\lambda_S\langle v_S \rangle \\
&=&\sum_{S \in\text{\rm supp}(v)}\mathrm{span}\,\{v_T\in\mathcal{B}\mid T\leq S\}\\
&=& \bigoplus_{S\in \mathcal{P}(v)} Fv_S, \quad\text{as subspaces}.
\end{eqnarray*}
This completes the proof of (2).

We now prove (3). It is trivial for $W=\{0\}$. Let W be a nonzero submodule of $V$. We claim that $W$ has a basis $\{v_S\in\mathcal{B}\mid S\in \mathcal{P}\}$ for some subset $\mathcal{P}$ of the power set of $\n$. Indeed, suppose $\mathcal{B}_1$ is a basis of $W$ and write every element of $\mathcal B_1$ as a linear combination of basis vectors in $\mathcal{B}=\{v_S\mid S\subseteq \n\}$. Let $\mathcal P$ be the set of all the different subsets $S$ where $S$ runs through the support of every element of $\mathcal B_1$. By (1) the set $\{v_S\in\mathcal{B}\mid S\in \mathcal{P}\}$ is a subset of $W$, and hence a basis of $W$ since it is linearly independent and spans $W$. Let
$
    w=\sum_{S\in\mathcal{P}} v_S.
$
By (1) again, $W$ is generated by $w$, and hence $W$ is cyclic.

We now show how to deduce a reduced generator of $W$ from $w$. Indeed, if $w$ contains two vectors $v_S, \,v_T$ with $T\le S$ and $T\ne S$ in supp($w$), then we can remove the term $v_T$ from $w$, and by Lemma \ref{mod1} (i) the sum of the remaining terms is still a generator. Repeat this process until we obtain the set
\[
    {\rm Red}(w) = \{S\mid  S \;\text{is maximal in supp}(w)\},
\]
and then we define the corresponding generator $w_{\rm red}$ of $W$ by
\[
    w_{\rm red}=\sum_{S\in{\rm Red}(w)}v_S~.
\]
We claim that $w_{\rm red}$ is a reduced generator of $W$. Let $v = \sum_{S\in\text{supp}(v)}\lambda_Sv_S$ be another generator of $W$. From Definition \ref{redGen} it suffices to show that $|{\rm supp(}v{\rm )}| \ge |{\rm Red(}w{\rm )}|$.
From (2) we find $W = \bigoplus_{T\in \mathcal{P}(v)} Fv_T = \bigoplus_{T\in \mathcal{P}(w)} Fv_T$ where $\mathcal{P}(v)$ and $\mathcal{P}(w)$ are as in (2),
and hence
$
    \mathcal{P}(v) = \mathcal{P}(w).
$
Define
\begin{equation}\label{redv}
    {\rm Red}(v) = \{S\mid  S \text{ is maximal in supp}(v)\}.
\end{equation}
Thus, ${\rm Red}(v) = \{S\mid S \text{ is maximal in }\mathcal{P}(v)\}$ and
${\rm Red}(w) = \{S\mid  S \text{ is maximal in }\mathcal{P}(w)\}$.
So, Red($v$) = Red($w$) and $|{\rm supp(}v{\rm )}| \ge |{\rm Red(}v{\rm )}| = |{\rm Red(}w{\rm )}|$, showing that $w_{\rm red}$ is reduced.

Suppose that $v = \sum_{S\in{\rm supp}(v)}v_S$ is another reduced generator of $W$. By the definition of reduced generators we know $|{\rm supp(}v{\rm )}| = |{\rm Red(}w{\rm )}|$. Hence $|{\rm supp(}v{\rm )}| =  |{\rm Red(}v{\rm )}|$ since Red($v$) = Red($w$). It follows that ${\rm supp(}v{\rm )} = {\rm Red(}v{\rm )}$. Let $v_{\rm red}=\sum_{S\in{\rm Red}(v)}v_S$. Then $v=v_{\rm red}=w_{\rm red}$. Therefore $w_{\rm red}$ is the unique reduced generator of $W$.
\end{proof}

\begin{definition} The set ${\rm Red}(v)$ in {\rm (\ref{redv})} is called the {\em reduced support} of $v$, and the element $v_{\rm red}=\sum_{S\in{\rm Red}(v)}v_S$ is termed the {\em reduced form} of $v$. The reduced support of $0$ is empty, and the reduced form of $0$ is itself.
\end{definition}

For example, if $n=7$ and $v = v_\emptyset - 2v_{\{1\}} + v_{\{3\}} + 5 v_{\{1, \,2\}} + 3v_{\{4,\, 7\}} - 2v_{\{5, \,6\}} + v_{\{1, \,2, \,3\}}$, then Red($v$) = $\{\emptyset,\,\{3\},\,\{5,\, 6\},\,\{4,\, 7\}, \{1, \,2, \,3\}\}$ is the reduced support of $v$, and its reduced form is $v_{\rm red} = v_\emptyset + v_{\{3\}} + v_{\{4,\, 7\}} + v_{\{5, \,6\}} + v_{\{1, \,2, \,3\}}$.

It is sometimes convenient to call the reduced support of $v$ the {\em reduced support} of the module $\langle v \rangle$.
A direct calculation yields that the reduced generator of $V_k$ is $v_{\{n-k+1,\,\ldots,\,n\}}$ for $1\le k\le n$, and the reduced support of $V_k$ is the set $\{n-k+1,\,\ldots,\,n\}$. The module $V_0$ has the element $v_\emptyset$ as its reduced generator, and its reduced support is the set $\{\emptyset\}$.

The next result is a consequence of Lemma \ref{mod1} (i) and Proposition \ref{cyclic} (3).

\begin{corollary}\label{eq}
    If $v, w\in V$, then $\langle v \rangle = \langle w \rangle$ if and only if they have the same reduced support {\rm Red(}v{\rm)} =  {\rm Red(}w{\rm)} if and only if they have the same reduced generator $v_{\rm red} = w_{\rm red}$.
\end{corollary}

\subsection{Dimensions of submodules of $V$}

We now describe the dimension of $\langle v_S \rangle$ for any $S\subseteq\n$.
\begin{theorem}\label{dim}
If $S=\{s_1<\dots<s_k\}$ is a $k$-subset of $\n$, let $d_S$ be the dimension of the module $\langle v_S \rangle$. If $k=1$ then $d_S = s_1$, and for $k\ge 2$,
\begin{align}
   d_S &= \sum^{k-1}_{i=1}{s_{k-i+1} \choose k+1-i}\gamma_i -\sum^{k-1}_{i=1}{s_{k-i+1}-s_1 \choose k+1-i}\gamma_i-\sum^{k-2}_{i=1}s_1 {s_{k-i+1}-s_2 \choose k-i}\gamma_i ~, \label{dk}
\end{align}
where $\gamma_1=1$ and for $2\leq j\leq k-1$,
\begin{equation*}\label{}
    \gamma_j=-\sum^{j-2}_{i=1}
              {s_{k+1-i}-s_{k+2-j} \choose j-i}\gamma_i ~.
    \end{equation*}
\end{theorem}
\begin{proof}
By Lemma \ref{mod1} (i) we know that $d_S$ is equal to the number of $k$-subsets $T$ of $\n$ such that $T\leq S$. Let
\begin{equation}\label{lai}
    \la_i=s_{k-i+1}-(k-i+1)\quad\text{ for }\quad 1\le i \le k~.
\end{equation}
Then $\lambda_i\ge\lambda_{i+1}$ since $s_{k-i+1}>s_{k-i}$.
Because the smallest $k$-subset is $\{1,\dots,k\}$, we have
\begin{equation}\label{lambdaSequence}
    \la_1\geq \dots\geq\la_k\geq 0~,
\end{equation}
and the number of $k$-subsets $T$ of $\n$ with $T\leq S$ is equal to the number $d_k$ of all the sequences
\begin{equation*}
    \mu_1\geq \dots\geq \mu_k\geq 0\quad\text{with}\quad \mu_i\leq \lambda_i\quad\text{for}\quad i=1,\dots,k~.
\end{equation*}
Thus $d_S = d_k$. From Lemma \ref{dkseq} and Theorem \ref{abg} we have
$d_1= \lambda_1+1$, and for $k\ge 2$,
{\small
\[
d_k = \sum^{k-1}_{i=1}\bigg[{\la_i+k-i+1 \choose k+1-i} -{\la_i-\la_k+k-i \choose k+1-i}\bigg]\gamma_i -\sum^{k-2}_{i=1}(\la_k+1) {\la_i-\la_{k-1}+k-i-1 \choose k-i}\gamma_i~,
\]
}
where
\begin{eqnarray*}
\gamma_1 = 1 \quad\text{and}\quad\gamma_j =-\sum^{j-2}_{i=1}
          {\la_i-\la_{j-1}+j-i-1 \choose j-i}\gamma_i~\quad\text{for } j\ge 2~.
\end{eqnarray*}
Using (\ref{lai}), we conclude that (\ref{dk}) holds.
\end{proof}

\begin{corollary}\label{dimspeci1}
If $S=\{2,4,\ldots,2k\} \subseteq \mathbf{n}$, then the dimension of the submodule $\langle v_S \rangle$ is the Catalan number
$c_{k+1}$.
\end{corollary}
\begin{proof}
           The sequence $\la_1\geq \dots\geq\la_k\geq 0$ in (\ref{lambdaSequence}) associated to $S$ is now $k \geq k-1\geq\cdots\geq 1$. In this case, it is well known that the number of sequences $\mu_1\ge\mu_2\ge\ldots\ge\mu_k$ of nonnegative integers such that $\mu_i\leq\la_i$ for $1\leq i\leq k$ is the Catalan number $c_{k+1}$. The desired result follows from Theorem \ref{dim}.
\end{proof}

\begin{corollary}\label{dimspeci2}
If the $k$-subset $S=\{m+1,\ldots,m+k\}\subseteq\mathbf{n}$, the
dimension of the submodule $\langle v_S \rangle$ is ${m+k \choose k}$.
\end{corollary}
\begin{proof}
   The sequence in (\ref{lambdaSequence}) corresponding to $S$ is the $k$-subset $\{m, \ldots,m\}$. A direct calculation of $d_S$ for $k\geq 1$ using the formulas given in Theorem \ref{dim} yields
    \begin{equation*}\label{}
        d_k={m+k \choose k}~,
    \end{equation*}
which is the desired result.
\end{proof}

We now compute the dimension $d_{k,\, m}$ of the $\mathcal{IC}_n$-module $\langle v_{S_{k,\,m}}\rangle$, where $k\ge m$ and
$$
    S_{k,\,m}=\{2,\,4,\,\ldots,\,2m,\,2m+1,\,2m+2,\,\ldots,\,2m+(k-m)\}
$$
is a subset of $\mathbf{n}$, which consists of both subsets in Corollaries \ref{dimspeci1} and \ref{dimspeci2}.
\begin{corollary}
For $m\ge 2$ the dimension of the module $\langle v_{S_{k,\,m}}\rangle$ is
\begin{eqnarray*}\label{catagene2}
 \nonumber d_{k,\,m}&=&{m+k \choose k}-{m+k-2 \choose k}-2
          {m+k-4 \choose k-1} +\sum^{k-1}_{i=k-m+3}{2(k-i+1) \choose k+1-i}\gamma_i \\
  &&{}-\sum^{k-1}_{i=k-m+3}{2(k-i) \choose k+1-i}\gamma_i-\sum^{k-2}_{i=k-m+3}2
          {2(k-i-1) \choose k-i}\gamma_i
\end{eqnarray*}
where $\gamma_1=1,\gamma_2=\gamma_3=\dots=\gamma_{k-m+2}=0,
\gamma_{k-m+3}=-1$, and for $i\geq k-m+4$

\begin{equation*}\label{catagene3}
\gamma_i =-{m-k+2i-4\choose i-1}-\sum^{i-2}_{j=k-m+3}
          {2(i-j-1) \choose i-j}\gamma_j~.
\end{equation*}
\end{corollary}
\begin{proof}
    The sequence (\ref{lambdaSequence}) associated to $S_{k,\,m}$ is $\{m,\, m, \, \ldots, m, \, m-1,\, m-2,\, \ldots,\,1\}$ of length $k$. The desired formula follows from Theorem \ref{dim}, and we omit the tedious calculation.
\end{proof}

\begin{remark}
    Notice that $d_{k,\,k}$ is exactly the Catalan number $c_{k+1}$ by Corollary (\ref{dimspeci1}).
\end{remark}

The intention below is to give another description of the dimensions of $\langle v_S \rangle$ using Proposition \ref{ck}.
\begin{theorem}\label{dim2}
If $S=\{s_1<\dots<s_k\}$ is a $k$-subset of $\n$ and $d_S$ is the
dimension of the module $\langle v_S \rangle$, then
\begin{equation}\label{dS}
    d_S=\sum\limits_{T=\{t_1<t_2<\dots<t_l\}\subseteq S}(-1)^{|S-T|}\det\limits_{1\leq i,j\leq l} \left(  \begin{array}{c}{t_i+1 \choose j-i+1} \\
            \end{array}
    \right)
\end{equation}
\end{theorem}
\begin{proof}

Let $\boldsymbol v$ denote the lattice path corresponding to $S$ in the closed first quadrant with starting point on the $y$-axis. Denote by $P(S)$ the set of all the lattice paths below $\boldsymbol v$, and identify each element of $P(S)$ with its height sequence.
We have
\[
    P(S)=\{p_1\leq\dots\leq p_k \mid 0\le p_i\leq s_i\text{ for } i=1,2,\dots,k\}~.
\]
To find the dimension of $\langle v_S\rangle$, by Lemma \ref{mod1} (i), it suffices to calculate the number of $k$-subsets less than or equal to $S$. Clearly each of these $k$-subsets, when regarded as a sequence, lies in $P(S)$, does not contain $p_1=0$ or repeated elements, but $P(S)$ contains sequences with $p_1=0$ or repeated elements.

Our goal is to distinguish from $P(S)$ all the sequences containing $p_1=0$ or repeated elements.
For any subset $S'=\{s_{i_1}<s_{i_2}<\dots<s_{i_m}\}\subseteq S$, let $I=\{i_1<i_2<\cdots<i_m\}\subseteq\bold k$ be the index set of $S'$, and let $T = S\setminus S'$. Define
\begin{equation*}
    P_I(S) = \{p_1\leq\dots\leq p_k \in P(S)\mid p_i=p_{i-1} \text{ for all } i\in I\},
\end{equation*}
where we agree that $p_{0}=0$ if $1\in I$, and in this case we must have $p_1=0$.
In other words, $P_I(S)$ consists of sequences $p_1\leq\dots\leq p_k$ in $ P(S)$ whose $i$th component $p_i$ is equal to its previous component $p_{i-1}$ for all $i\in I$.

There is a one-to-one correspondence between $P_I(S)$ and $P(T)$. Indeed, let $p_{j_1}\leq p_{j_2}\leq\dots\leq p_{j_l}$ where $l=k-m$ be obtained from $p_1\leq\dots\leq p_k$ by removing $p_i$ for all $i\in I$. Then the map
$f:P_I(S)\rightarrow P(T)$ defined by
\[
    f(p_1\leq\dots\leq p_k) = p_{j_1}\leq p_{j_2}\leq\dots\leq p_{j_l}
\]
is bijective.
Formula (\ref{dS}) follows from Proposition \ref{ck} and
the inclusion-exclusion principle.
\end{proof}

We now describe the dimension of any nonzero submodule of $V$. Proposition \ref{cyclic} (3) assures that the submodule is equal to the module $\langle v \rangle$ generated by some $v\in V$.
\begin{theorem}\label{dimEvery}
Let $v\in V$ and {\rm Red(}v{\rm )}= $\{S_1,\,\ldots,\,S_m\}$. For any $J\subseteq$ {\rm Red(}v{\rm )} denote by $S_J$ the greatest lower bound of $\{S_j\mid j\in J\}$. Then the dimension of $\langle v \rangle$ is given by
\[
    \dim \langle v \rangle = \sum_{\emptyset\,\ne J\,\subseteq {\bold m}} (-1)^{|J|-1} \dim \langle v_{S_J}\rangle.
\]
\end{theorem}
\begin{proof}
  From Proposition \ref{cyclic} (2) and (3) the dimension of $\langle v \rangle$ is equal to the cardinality of the set
  $\mathcal{P}(v) = \bigcup_{S\in {\rm Red(}v{\rm )}}\{T\subseteq\n\mid T\le S\}$. Let $A_j=\{T\subseteq\n\mid T\le S_j\}$, $j\in{\bold m}$. Then $\mathcal{P}(v) = \bigcup_{j\in{{\bold m}}}A_j$, and $\dim \langle v_{S_j} \rangle = |A_j|$ by Proposition \ref{cyclic} (2). With Lemma \ref{mod1} (3) in mind and applying the inclusion-exclusion principle to count the cardinality of $\mathcal{P}(v)$, we obtain the desired formula for $\dim \langle v \rangle$.
\end{proof}

\vspace{5mm}
\noindent Jianqiang Feng, ~Wenli Liu, ~Ximei Bai

\vspace {-1mm}
\noindent College of Mathematics and Information Science, Hebei University, Baoding, 071002; Email: vonjacky@126.com, liuwl$\_$hb@163.com, and baixm131@163.com

\vspace{5mm}
\noindent Zhenheng Li

\vspace {-1mm}
\noindent College of Mathematics and Information Science, Hebei University, Baoding, 071002, and Department of Mathematical Sciences, University of South Carolina Aiken, SC 29803; Email: zhenhengl@usca.edu

\end{document}